\documentclass[12pt]{amsart}
\usepackage{amssymb}
\input amssym.def
\usepackage{amsmath,amsfonts,hyperref,xcolor,textcomp}
\usepackage{amscd}
\usepackage[mathscr]{eucal}
\hypersetup{colorlinks=true,citecolor={purple},linkcolor={blue},urlcolor={violet}}

\newcommand{\N}{{\mathbb N}}
\newcommand{\Z}{{\mathbb Z}}
\newcommand{\Q}{{\mathbb Q}}
\newcommand{\C}{{\mathbb C}}

\newcommand{\bQ}{\overline{\mathbb Q}}
\newcommand{\cA}{{\mathcal A}}

\newcommand{\cC}{{\mathcal C}}

\newcommand{\cE}{{\mathcal E}}
\newcommand{\cP}{{\mathcal P}}
\newcommand{\cL}{{\mathcal L}}

\newcommand{\trdeg}{\text{\rm trdeg}}

\newtheorem{thm}{Theorem}
\newtheorem{lem}{Lemma}
\newtheorem{cor}{Corollary}
\newtheorem{conj}{Conjecture}

\newtheorem{rmk}{Remark}
\newtheorem{defn}{Definition}

\newcommand{\thmref}[1]{Theorem~\ref{#1}}

\newcommand{\lemref}[1]{Lemma~\ref{#1}}
\newcommand{\corref}[1]{Corollary~\ref{#1}}
\newcommand{\conjref}[1]{Conjecture~\ref{#1}}

\begin{document}

\title{An abelian analogue of Schanuel's conjecture and applications}

\author{Patrice Philippon, Biswajyoti Saha and Ekata Saha}

\address{Patrice Philippon \\ \newline
\'Equipe de Th\'eorie des Nombres, Institut de Math\'ematiques
de Jussieu-Paris Rive Gauche, UMR CNRS 7586, Paris, France.}
\email{patrice.philippon@upmc.fr}

\address{Biswajyoti Saha\\ \newline
School of Mathematics and Statistics, University of Hyderabad,
Prof. C.R. Rao Road, Gachibowli, Hyderabad 500046, India.}
\email{biswa.imsc@gmail.com, biswa@uohyd.ac.in}

\address{Ekata Saha\\ \newline
Statistics and Mathematics Unit, Indian Statistical Institute,
203 Barrackpore Trunk Road, Kolkata 700 108, India.}
\email{ekata.imsc@gmail.com, ekata@math.tifr.res.in}



\date{September 4, 2019}

\begin{abstract}
In this article we study an abelian analogue of Schanuel's conjecture. This
conjecture falls in the realm of the generalised period conjecture of Y. Andr\'e.
As shown by C. Bertolin, the generalised period conjecture includes
Schanuel's conjecture as a special case. Extending methods of Bertolin,
it can be shown that the abelian analogue of Schanuel's conjecture we consider,
also follows from Andr\'e's conjecture.  C. Cheng et al. showed that the classical Schanuel's
conjecture implies the algebraic independence of the values  of the iterated exponential
function and the values of the iterated logarithmic function, answering a question of
M. Waldschmidt. We then investigate a similar question in the setup of abelian varieties.
\end{abstract}

\maketitle

\section{Introduction}

S.~Schanuel proposed the following conjecture while attending a course given by
S.~Lang at Columbia University in the 1960's. Most of the known results in the
transcendental number theory about the values of the exponential function are
encompassed by Schanuel's conjecture, and they can be derived as its consequence.

\begin{conj}[Schanuel]\label{Schanuel}
Let $x_1,\ldots,x_n \in \C$ be such that they are linearly independent
over $\Q$. Then the transcendence degree of the field
$$
\Q(x_1,\ldots,x_n,e^{x_1},\ldots,e^{x_n})
$$
over $\Q$ is at least $n$.
\end{conj}

For example, C.~Cheng et al. have shown in \cite{CDHHKMMMW} how to derive from
\conjref{Schanuel}, the linear disjointness of the two fields constructed over
$\Q$ by adjoining repeatedly the algebraic closure of the field generated by the values
of the exponential and logarithm functions respectively.

The only known cases of \conjref{Schanuel} are $n=1$ and
$x_1,\ldots,x_n \in \bQ$ for general $n$. The $n=1$ case is a consequence
of the Hermite-Lindemann theorem, whereas the latter case is known as
the Lindemann-Weierstrass theorem. But these two special cases were
known much before the inception of this conjecture. 

Schanuel's conjecture has been generalised to various other
contexts. The elliptic analogue of Schanuel's conjecture is well-studied.
Let $\Lambda$ be a lattice in $\C$ and $\wp$ denote the associated
Weierstrass $\wp$-function,
$$
\wp(z)=\wp(\Lambda;z):= \frac{1}{z^2}+ \sum_{\omega \in \Lambda \setminus \{0\}}
\left( \frac{1}{(z-\omega)^2} - \frac{1}{\omega^2} \right)
$$
for $z \in \C \setminus \Lambda$. The Weierstrass $\wp$-function is an
elliptic function with double poles at the points of $\Lambda$ and
holomorphic in $\C \setminus \Lambda$. Moreover, for all
$z \in \C \setminus \Lambda$, we have the relation
$$
\wp'(z)^2=4 \wp(z)^3 - 60 G_4(\Lambda) \wp(z) - 140 G_6(\Lambda).
$$
Here for $k \ge 2$,
$G_{2k}(\Lambda):= \sum_{\omega \in \Lambda \setminus \{0\}} \omega^{-2k}$
is the associated Eisenstein series of weight $2k$.
Let $g_2=60 G_4(\Lambda)$ and $g_3=140 G_6(\Lambda)$.
Then the modular invariant $j(\Lambda)$ is defined by
$$
j(\Lambda):=1728 \frac{g_2^3}{g_2^3-27g_3^2}.
$$
and the associated Weierstrass $\zeta$-function is defined by
$$
\zeta(z)=\zeta(\Lambda;z):= \frac{1}{z}+ \sum_{\omega \in \Lambda \setminus \{0\}}
\left( \frac{1}{z-\omega} + \frac{1}{\omega}+ \frac{z}{\omega^2} \right),
$$
where the series above converges absolutely and uniformly in any compact subset of
$\C \setminus \Lambda$. Thus it is holomorphic  in $\C \setminus \Lambda$
and $\zeta'(z)=-\wp(z)$. If $\omega_1,\omega_2$ denote the fundamental periods of $\Lambda$, then
the quasi-periods $\eta_1,\eta_2$ are defined by $\eta_i:=\zeta(z+\omega_i)-\zeta(z)$
for $i=1,2$. With these notations, the elliptic Schanuel conjecture reads as follows (see \cite{CB}) :

\begin{conj}[elliptic Schanuel]\label{eS}
Let $\Lambda$ be a lattice and $\wp,\zeta$ denote the associated
Weierstrass functions. Let $K$ be the field of endomorphisms of $\Lambda$
and $x_1,\ldots,x_n \in \C \setminus \Lambda$ such that they are
linearly independent over $K$. Then
$$
\trdeg_\Q \Q(g_2,g_3,\omega_1,\omega_2,\eta_1,\eta_2, x_1,\ldots,x_n,
\wp(x_1),\ldots,\wp(x_n),\zeta(x_1),\ldots,\zeta(x_n)) \ge 2n+ \frac{4}{[K:\Q]}.
$$
\end{conj}

Often the weaker statement
\begin{equation}\label{weak-elliptic}
\trdeg_\Q \Q(g_2,g_3, x_1,\ldots,x_n, \wp(x_1),\ldots,\wp(x_n)) \ge n
\end{equation}
is also considered for application of the elliptic Schanuel conjecture.
Here also the $n=1$ case is known and it can be deduced as a consequence of a
more general theorem of T. Schneider and S. Lang about transcendental values
of meromorphic functions. The analogue of Lindemann-Weierstrass theorem when
the Weierstrass $\wp$-function with algebraic invariants $g_2,g_3$
has complex multiplication, was proved independently by the first author
\cite{PP} and G. W\"ustholtz \cite{GW}.

We also have A. Grothendieck's period conjecture for an abelian variety $A$, defined
over $\bQ$. It states that the transcendence degree of the period matrix of $A$ is
the same as the dimension of the associated Mumford-Tate group. From the work
of P. Deligne \cite[Cor. 6.4, p.76]{PD} one gets that this dimension is at least an upper bound
for the transcendence degree of the period matrix. Y. Andr\'e \cite[Chap. 23]{YA} suggested
a more general conjecture which is now known as the generalised period conjecture.
It predicts a lower bound for the transcendence degree of the field
generated by the periods of a mixed motive defined over an algebraically closed
sub-field $K$ of $\C$, which is not necessarily algebraic over $\Q$. 
In fact, Grothendieck's period conjecture can be seen as
a special case of Andr\'e's conjecture, using Deligne's work. Further,
C. Bertolin \cite{CB} has shown that this generalised period conjecture,
applied to 1-motives defined over a subfield of $\C$, not necessarily algebraic, includes
Schanuel's conjecture as a special case.

We now consider the following weaker version of Andr\'e's generalised period
conjecture. Let $A(\C)$ be an abelian variety of dimension $g$ defined over $\bQ$
and $\exp_A: \C^g \to A(\C)$ denote the exponential map, 
which is periodic with respect to a lattice $\Lambda_A$.
Let $\omega_1,\dots,\omega_g$ be a basis of the holomorphic
differential 1-forms and $\eta_1,\dots,\eta_g$ be a basis of the meromorphic
differential 1-forms with residue $0$ on $A$. Next let
$\gamma_1,\dots,\gamma_{2g}$ be a basis of the homology of $A$.
So the matrix of period $\tilde\Lambda_A$ is the $2g\times 2g$ matrix with entries
$\int_{\gamma_j}\omega_i$ and $\int_{\gamma_j}\eta_i$, $i=1,\dots,g$, $j=1,\dots,2g$,
while the matrix of the lattice $\Lambda_A$ is the $g\times 2g$ matrix with entries $\int_{\gamma_j}\omega_i$.

Let $u\in{\C}^g$ and $y = \exp_A(u)$. The relevant 1-motive here is $M = [{\Z} \to A]$, $1 \mapsto y$,
which is defined over ${\bQ}(\exp_A(u))$. Let $MT(M)$ denote its Mumford-Tate group.
The periods of the 1-motive $M$ include the periods of $A$ and the components of $u$
i.e. the numbers $\int_0^u\omega_i$, and also the integrals $\int_0^u\eta_i$. Let
$\zeta_A(u)$ denote the vector with components $\int_0^u\eta_i$, $i=1,\dots,g$.

\begin{conj}[weak abelian Schanuel]\label{aS}
With the notations as above,
let $\bQ(\tilde\Lambda_A)$ denote the field generated by the periods
and quasi-periods over $\bQ$. Let $u\in{\C}^g$ and $H$ be the smallest
algebraic subgroup of $A$ containing the point $\exp_A(u)$. Then
\begin{equation}\label{qp-2}
\trdeg_{\bQ(\tilde\Lambda_A)} \bQ(\tilde\Lambda_A,\exp_{A}(u),u,\zeta_A(u)) \ge 2\dim(H).
\end{equation}
\end{conj}

In a discussion, D. Bertrand told us that G. Vall\'ee \cite{GV} deduces
\conjref{aS} from Andr\'e's conjecture based on \cite[Proposition 1]{YA1}
and extending the methods of \cite{CB}. With due consent, we
reproduce an indication of this argument here.

The generalised period conjecture of Andr\'e for $M$ implies 
$$
\trdeg_{\bQ}\bQ(\tilde\Lambda_A,\exp_{A}(u),u,\zeta_A(u)) \ge \dim MT(M).
$$
Since $A$ is defined over $\bQ$, Grothendieck's conjecture (which is a particular case
of Andr\'e's conjecture) gives $\trdeg_{\bQ}\bQ(\tilde\Lambda_A) = \dim(MT(A))$. Hence
\begin{equation}\label{tr-deg}
\trdeg_{\bQ(\tilde\Lambda_A)} \bQ(\tilde\Lambda_A,\exp_{A}(u),u,\zeta_A(u))
\ge \dim MT(M) - \dim MT(A).
\end{equation}

If $U(M)$ denotes the unipotent radical of $MT(M)$, then $MT(M)/U(M)$ is the
(reductive) group $MT(A)$. Hence the right hand side of \eqref{tr-deg}
equals $\dim(U(M))$. Furthermore, by \cite[Proposition 1]{YA1}, $U(M)$ is known
to be equal to $H^1_{Betti}(H^\circ)$, where $H^\circ$ is the connected
component of $H$ containing the trivial element. Thus 
$\dim(U(M))=2\dim(H)$ and we therefore get \eqref{qp-2}, i.e.
\begin{equation*}
\trdeg_{\bQ(\tilde\Lambda_A)} \bQ(\tilde\Lambda_A,\exp_{A}(u),u,\zeta_A(u)) \ge 2\dim(H).
\end{equation*}
Thus \conjref{aS} follows from the generalised period conjecture of Andr\'e.

\begin{rmk}\label{rmk-aS-qp}\rm
Note that
$$
\trdeg_{\bQ(\tilde\Lambda_A)} \bQ(\tilde\Lambda_A,\exp_{A}(u),u,\zeta_A(u))
\le \trdeg_{\bQ(\tilde\Lambda_A)} \bQ(\tilde\Lambda_A,\exp_{A}(u),u) + \dim(H).
$$
Hence from \eqref{qp-2} we can deduce, 
\begin{equation}\label{qp-1}
\trdeg_{\bQ(\tilde\Lambda_A)} \bQ(\tilde\Lambda_A,u, \exp_A(u)) \ge \dim(H),
\end{equation}
and therefore
\begin{equation}\label{qp-0}
\trdeg_{\bQ(\Lambda_A)} \bQ(\Lambda_A, u, \exp_A(u)) \ge \dim(H).
\end{equation}
So we get \eqref{qp-1} and \eqref{qp-0} as consequence of \eqref{qp-2}.
\end{rmk}

\begin{rmk}\rm
\conjref{aS} can further be considered for abelian varieties defined over a
subfield of $\C$, not necessarily algebraic. G. Vall\'ee \cite{GV} has formulated the
relevant statement from Andr\'e's generalised period conjecture, and his statement
includes \conjref{aS} as a special case.
\end{rmk}

We therefore have supporting evidence for considering \conjref{aS}. 
With this conjecture in place we want to extend the results of C.~Cheng
et al.~\cite{CDHHKMMMW} to this setting, that is prove the linear disjointness of
the two fields defined below. For the sake of completeness we recall the definition
of linear disjointness (see \cite[Chap. VIII, \S3]{SL}). 

\begin{defn}
Let $F$ be a field and $F_1,F_2$ two of its field extensions contained
in a larger field $G$. Then $F_1$ is said to be linearly disjoint ({\it resp.} free) from 
$F_2$ over $F$ if any finite $F$-linearly ({\it resp.} algebraically) independent subset
of $F_1$ is also $F_2$-linearly ({\it resp.} algebraically) independent (as a subset of $G$).
\end{defn}

Though the above definition is asymmetric, it can be shown that the property
of being linearly disjoint ({\it resp.} free) is actually symmetric for $F_1$ and $F_2$.
It is easy to see that if $F_1$ and $F_2$ are linearly disjoint over $F$ then
$F_1 \cap F_2 = F$. Also if $F_1$ and $F_2$ are linearly disjoint over $F$ then
one can deduce that $F_1$ and $F_2$ are free over $F$
(see \cite[Chap. VIII, Prop. 3.2]{SL}). The converse is true in special
cases (see \cite[Chap. VIII, Theorem 4.12]{SL} and \lemref{l3} below).
The property of being free is also called as $F_1$ and $F_2$
being algebraically independent over $F$.

We now setup the relevant notations for our theorem. Recall, $A$ is an abelian
variety over $\bQ$ and $\tilde\Lambda_A$ denote the matrix of periods.
We consider two recursively defined sets $\cE,\cL$. Let us define
\footnote[1]{As previously for the field $\Q$, we denote with a bar $\overline{K}$
 the algebraic closure in $\C$ of a subfield $K\subset\C$.}
$$
\cE =\bigcup_{n \ge 0} \cE_n \ \ \text{and} \ \ \cL =\bigcup_{n \ge 0} \cL_n
$$
where $\cE_0=\bQ$, $\cL_0 = \bQ$ and for $n \ge 1$,
$$
\cE_n = \overline{\cE_{n-1}(\{\text{components of }\exp_A(u) : u \in \cE_{n-1}^g\})}
$$
and
$$
\cL_n = \overline{\cL_{n-1}(\{\text{components of }u, \zeta_A(u) : \exp_A(u) \in A(\cL_{n-1})\})}.
$$
Often, as in \cite{GV}, for a point $P$ of $A(\C)$, one uses $\log_A P$
({\it resp.}  $\tilde\log_A P$) to denote the point
$u=\exp_A^{-1}(P)$ ({\it resp.} $(u,\zeta_A(u))$). By convention we take
$\exp_A(\tilde\log_A P)=\exp_A(\log_A P)=P$. Given a suitable set $S$,
we denote $\exp_A(S)$ ({\it resp.} $\log_A(S)$, $\tilde\log_A(S)$) the set of
elements $\exp_A(u)$ for $u\in S$ ({\it resp.} $\log_A(P)$, $\tilde\log_A(P)$
for $P\in S$). Then, by induction, one can see that for $n \ge 1$,
$$
\cE_n = \overline{\Q(\text{components of } \exp_A(\cE_{n-1}^g))}
\ \ \text{and} \ \
\cL_n = \overline{\Q(\tilde \Lambda_A,\text{components of } \tilde\log_A(A(\cL_{n-1})))}.
$$
Below we state our main theorem about $\cE$ and $\cL$,
where we take the field $G$ in Definition 1 to be $\C$.

\begin{thm}\label{lin-dis}
If \conjref{aS} is true, then $\cE$ and $\cL$ are linearly disjoint over $\bQ$.
Since $\cE,\cL$ are algebraically closed, it is equivalent to say that they are
algebraically independent over $\bQ$.
\end{thm}

Proof of this theorem follows the structure of the proof of the main theorem of
\cite{CDHHKMMMW}, but a good part of it differs towards the end of our proof.
It will be interesting to consider similar problem for semi-abelian varieties
to encompass the cases treated here and in \cite{CDHHKMMMW}.

%\begin{rmk}
%\rm Similar constructions can be adapted with the base field $\bQ(\tilde\Lambda_A)$,
%in place of $\bQ(\Lambda_A)$. In that case we would have such a linear disjointness
%result as a consequence of \conjref{aS-qp} and the proof works verbatim as of \thmref{lin-dis},
%with $\bQ(\Lambda_A)$ replaced by $\bQ(\tilde\Lambda_A)$. If we want to use \eqref{qp-2}
%in place of \eqref{qp-1}, we have to modify our construction to include $\zeta_A(u)$, and
%the proof works with necessary minimal modifications.
%\end{rmk}

\begin{rmk}\rm
In \thmref{lin-dis}, one can consider abelian varieties defined over a
subfield of $\C$, not necessarily algebraic, as considered by G. Vall\'ee \cite{GV}.
However, the statement is not true as it stands. In \S4.1, we exhibit an elliptic curve
such that $\cE_1 \cap \cL_1 \supsetneq \cE_0 \cap \cL_0$ for the natural
candidate of $\cE_0$ and $\cL_0$. The difficulty in this case is
coming from the fact that we no more have the equality in Grothendieck's period
conjecture (also see \cite[Chap. 23.4]{YA}).
\end{rmk}

\begin{rmk}
\rm We get an immediate application
of \thmref{lin-dis} for elliptic curves over $\bQ$. See \S4 for more details.
\end{rmk}

\section{Intermediate lemmas}

In this section we deduce some intermediate results  to prove \thmref{lin-dis}.

\begin{lem}\label{l3}
Let $K_1,K_2$ be two sub-fields of $\C$, which are algebraically closed over $K_1 \cap K_2$.
Then they are algebraically independent over $K_1 \cap K_2$ if and only if they are linearly
disjoint over $K_1 \cap K_2$.
\end{lem}

\begin{proof}
It follows immediately from Theorem 4.12 and Proposition 3.2 of \cite[Chap. VIII]{SL}.
\end{proof}

From now on in this section, $n$ denotes an integer $\ge 1$.

\begin{lem}\label{l1}
Let $A_n$ be a finite subset of $\cE_n$. Then there exists
a finite subset $\cA$ of $\cE_{n-1}$ such that
$\cA \cup A_n$ is algebraic over $\bQ(\text{components of }\exp_A(\cA^g))$.
\end{lem}

\begin{proof}
Since $A_n \subset \cE_n$, for each $x \in A_n$, there exists a finite subset $C_x$ of
$\cE_{n-1}$ such that $x$ is algebraic over $\bQ(\text{components of }\exp_A(C_x^g))$.
Let $A_{n-1}:= \cup_{x \in A_n} C_x \subset \cE_{n-1}$. Then $A_{n-1}$ is a finite
subset of $\cE_{n-1}$.

We repeat the process to get sets $\{A_i\}_{0 \le i \le n-2}$
such that for each $i$, $A_i$ is a finite subset of $\cE_i$ and $A_{i+1}$ is algebraic
over $\bQ(\text{components of }\exp_A(A_i^g))$. We take
$\cA:= \cup_{0 \le i \le n-1} A_i \subset \cE_{n-1}$. Then
$\cA \cup A_n$ is algebraic over $\bQ(\text{components of }\exp_A(\cA^g))$.
\end{proof}

\begin{lem}\label{l2}
Let $C$ be a finite subset of $\cL_n$. Then there exists
a finite set $\cC \subset \cL_{n}^{g}$ with $\exp_A(\cC) \subset A(\cL_{n-1})$
such that the set $\{\text{components of }\exp_A(\cC)\} \ \cup \ C$ is algebraic over
the field $\bQ(\tilde\Lambda_A,\text{components of }\cC)$.
\end{lem}

\begin{proof}
Since $C \subset \cL_n$, for each $y \in C$, there exists a finite subset $D_y$ of
$A(\cL_{n-1})$ such that $y$ is algebraic over
$\bQ(\tilde\Lambda_A,\text{components of }\tilde\log_A(D_y))$.
Define $B_{n-1}:=\tilde\log_A(\cup_{y \in C}D_y)$, so that
$C$ is algebraic over $\bQ(\tilde\Lambda_A,\text{components of }B_{n-1})$.
Then $B_{n-1}\subset \tilde\log_A(A(\cL_{n-1})) \subset \cL_{n}^{2g}$.
Hence $\exp_A(B_{n-1}) \subset A(\cL_{n-1})$,
i.e. the components of $\exp_A(B_{n-1})$ is a finite subset of $\cL_{n-1}$.

We repeat this process for the components of $\exp_A(B_{n-1})$
in place of $C$ and so on, to get sets $\{B_i\}_{0 \le i \le n-2}$
such that for each $i$, $\exp_A(B_{i}) \subset A(\cL_{i})$ and components of
$\exp_A(B_{i+1})$ is algebraic over $\bQ(\tilde\Lambda_A,\text{components of }B_i)$.
We set $\cC:= \cup_{0 \le i \le n-1} B_i$ to complete the proof.
\end{proof}

\section{Proof of \thmref{lin-dis}}

In view of \lemref{l3}, we show that $\cE$ and $\cL$ are algebraically independent over $\bQ$.
It is enough to prove that $\cE_m$ and $\cL_n$ are algebraically independent
over $\bQ$ for all $m,n$.

Now suppose that there exists a pair
$(m,n) \in \N^2$ such that $\cE_m$ and $\cL_n$ are not algebraically
independent over $\bQ$. We choose such a pair $(m,n)$ with the property that
if $(a,b) < (m,n)$, then $\cE_a$ and $\cL_b$ are algebraically independent over $\bQ$.
Here the ordering `$<$' is the ordering on $\N^2$ where $(a,b) < (m,n)$ if and only if either
$a \le m$ and $b < n$, or $a < m$ and $b \le n$. Clearly $m,n \ge 1$.

As $\cE_m$ and $\cL_n$ are not algebraically independent over $\bQ$,
there exist elements $\ell_1, \ldots, \ell_h$ of $\cL_n$, algebraically independent
over  $\bQ$ which are algebraically dependent over $\cE_m$ i.e. there exists
a finite subset $\{e_1,\ldots,e_k\}$ of elements of $\cE_m$ 
such that $\ell_1,\ldots, \ell_h$ are algebraically dependent over $\bQ(e_1, \ldots, e_k)$.

Now from \lemref{l1}, we know that there exists
a finite subset $\cA$ of $\cE_{m-1}$ such that
$\cA \cup \{e_1, \ldots, e_k\}$ is algebraic over
$\bQ(\text{components of }\exp_A(\cA^g))$.
Similarly by \lemref{l2}, we have a finite set
$\cC \subset \cL_{n}^{2g}$ with $\exp_A(\cC) \subset A(\cL_{n-1})$
such that $\{\text{components of }\exp_A(\cC)\} \ \cup \ \{ \ell_1, \ldots, \ell_h \}$
is algebraic over $\bQ(\tilde\Lambda_A,\text{components of }\cC)$.
Let $|\cA^g|=n_1$ and $|\cC|=n_2$.

Define ${\bf u_1}:=(u : u \in \cA^g) \in TA^{n_1}$
by concatenating elements of $\cA^g$ one after another. Here $TA^{n_1}$
denotes the tangent space of the abelian variety $A^{n_1}$. Let $A_1$
be the smallest algebraic subgroup of $A^{n_1}$
containing the point $\exp_{A^{n_1}}({\bf  u_1})$.
Similarly define ${\bf u_2} := (u : u \in \mathcal C) \in TA^{n_2}$.
%${\bf u_2} := (u : (u,\zeta_A(u)) \in \mathcal C) \in TA^{n_2}$.
Let $A_2$ be the smallest algebraic subgroup of $A^{n_2}$
containing the point $\exp_{A^{n_2}}({\bf  u_2})$.

Let
$$
K_1:=\overline{\Q({\bf  u_1},\exp_{A^{n_1}}({\bf  u_1}),\zeta_{A^{n_1}}({\bf  u_1}))}
\ \text{and} \
K_2:=\overline{\Q(\tilde\Lambda_A,{\bf  u_2},\exp_{A^{n_2}}({\bf  u_2}),\zeta_{A^{n_2}}({\bf  u_2}))}
$$
Then $\bQ(e_1, \ldots, e_k) \subset K_1 =
\overline{\Q(\exp_{A^{n_1}}({\bf  u_1}),\zeta_{A^{n_1}}({\bf  u_1}))}$ by \lemref{l1}
and $\bQ(\ell_1,\ldots,\ell_h) \subset K_2 =
\overline{\Q(\tilde\Lambda_A,{\bf  u_2},\zeta_{A^{n_2}}({\bf  u_2}))}$
by \lemref{l2}. Hence $K_1$ and $K_2$ are not algebraically independent over $\bQ$. However,
by \conjref{aS}, we get
$$
\trdeg_{\bQ} K_1 \ge 2\dim(A_1) \ \text{and} \ \trdeg_{\bQ(\tilde\Lambda_A)} K_2 \ge 2\dim(A_2)
$$
for $i=1,2$. Since $\exp_{A^{n_1}}({\bf  u_1}) \in A_1$
and ${\bf  u_2} \in TA_2$, we get that 
$$
\trdeg_{\bQ} K_1 \le 2\dim(A_1)
$$
and
$$
\trdeg_{\bQ(\tilde\Lambda_A)} K_2 \le 2\dim(TA_2) = 2\dim(A_2).
$$
%by \lemref{l1}, \lemref{l2} respectively.
Hence
$$
\trdeg_{\bQ} K_1 = 2\dim(A_1) \ \text{and} \ \trdeg_{\bQ(\tilde\Lambda_A)} K_2 = 2\dim(A_2).
$$
We want to show that
\begin{equation}\label{to-prove}
\trdeg_{\bQ(\tilde\Lambda_A)} K_1 K_2 = 2\dim(A_1 \times A_2)
=\trdeg_{\bQ} K_1+\trdeg_{\bQ(\tilde\Lambda_A)} K_2.
\end{equation}
Adding $\trdeg_{\bQ} \bQ(\tilde\Lambda_A)$ to both sides of \eqref{to-prove}, we would get
$$
\trdeg_{\bQ} K_1 K_2 =\trdeg_{\bQ} K_1+\trdeg_{\bQ} K_2.
$$
This would prove that the fields $K_1$ and $K_2$ are algebraically
independent over $\bQ$. We will thus get a contradiction to our assumption.

Define ${\bf  u_3}:=({\bf  u_1},{\bf  u_2})$.
Let $B$ be the smallest algebraic subgroup of $A^{n_1+n_2}$
containing the point $\exp_{A^{n_1+n_2}}({\bf  u_3})$.
Then by \conjref{aS}, we get
$$
\trdeg_{\bQ(\tilde\Lambda_A)} K_1 K_2 \ge 2\dim(B).
$$
Thus we are reduced to prove that $\dim(B)=\dim(A_1) + \dim(A_2)$.
If they are torsion subgroups then we have nothing to prove.
So we assume that at least one of $A_1$ and $A_2$ is not a torsion subgroup.

We first assume $A$ to be simple. For ${\bf  u_i} \in TA_i \hookrightarrow TA^{n_i}$,
we choose a basis and write ${\bf  u_i}=(u_{i1},\ldots,u_{in_i})$ with $u_{ij} \in TA$
for $i=1,2$ and $j=1,\ldots,n_i$. We consider any of the defining relation
for $TB$,
\begin{equation}\label{def-TB}
\sum_{1\le j \le n_1} \delta_{1j} x_{1j} - \sum_{1 \le j \le n_2} \delta_{2j} x_{2j}=0,
\end{equation}
where $\delta_{ij} \in End(A)$ for $i=1,2$ and $j=1,\ldots,n_i$. Let $u:=
\sum_{1\le j \le n_1} \delta_{1j} u_{1j} = \sum_{1 \le j \le n_2} \delta_{2j} u_{2j}$.
Then $u \in \cE_{m-1}^g \cap \cL_{n}^g$ as each $u_{1j} \in \cA^g$
and $u_{2j} \in \cC$. Thus $u \in \bQ^g$, by the choice of $m,n$.

On the other hand $\exp_A(u)=\sum_{1\le j \le n_1} \delta_{1j} \exp_A(u_{1j})
= \sum_{1 \le j \le n_2} \delta_{2j} \exp_A(u_{2j})$. For similar reason
$\exp_A(u) \in A(\cE_m) \cap A(\cL_{n-1})$. Again $\cE_m \cap \cL_{n-1} =\bQ$,
and hence $\exp_A(u) \in A(\bQ)$. Thus,
$$
\trdeg_{\bQ} \bQ(u, \exp_A(u))=0
$$
Now if $H$ is the smallest algebraic subgroup of $A$ containing $\exp_A(u)$,
then $\dim(H)=0$, by \conjref{aS}, i.e. $H$ is torsion subgroup.
In particular, $n u \in \bQ^g$ and $\exp_A(nu)=0$, for a suitable integer $n$.
%Without loss of generality we may assume $n=1$.

Thus, $\exp_A(nu)=\sum_{1\le j \le n_1} \delta_{1j} \exp_A(nu_{1j})
= \sum_{1 \le j \le n_2} \delta_{2j} \exp_A(nu_{2j})=0$. Now
$A_i$ is the smallest algebraic subgroup of $A^{n_i}$
containing the point $\exp_{A^{n_i}}({\bf  u_i})$ for $i=1,2$.
Thus, $\sum_{1\le j \le n_i} \delta_{ij} \exp_A(nx_{ij})=0$
for any point $\exp_{A^{n_i}}({\bf  x_i}) \in A_i$ for $i=1,2$.
Since at least one of $A_1$ or $A_2$ is not a torsion subgroup,
there exists $i\in\{1,2\}$, such that $\sum_{1\le j \le n_i} \delta_{ij} n x_{ij}=0$ on $TA_i$.
%Hence for $i=1,2$, either $\sum_{1\le j \le n_i} \delta_{ij} n x_{ij}=0$ in $TA_i$ or
%$A_i$ is a torsion subgroup. If $\sum_{1\le j \le n_i} \delta_{ij} n x_{ij}=0$ in $TA_i$
Hence $\sum_{1\le j \le n_i} \delta_{ij} x_{ij}=0$ on $TA_i$, and therefore
defining relations for $TB$ separate into disjoint relations defining $TA_1$
and $TA_2$. Thus we have $\dim(B)=\dim(A_1) + \dim(A_2)$.

Now we treat the case when $A$ is not a simple abelian variety. In this case,
we would like to write down the generic form of a defining relation for $TB$
and we show that it is a collection of relations of the form \eqref{def-TB}.
Then the proof will follow as above.

We suppose that $A^{n_1+n_2}$ is isogenous to $V_1^{r_1} \times \cdots \times V_l^{r_l}$,
where for $1 \le i \neq j \le l$, $V_i$ is an abelian variety not isogenous to $V_j$. Thus,
the tangent space $TA^{n_1+n_2}$ has the form
$$
\underbrace{TV_1\oplus \cdots \oplus TV_1}_{r_1 \text{ times}}
\oplus \cdots \oplus \underbrace{TV_l\oplus \cdots \oplus TV_l}_{r_l \text{ times}}.
$$
Now $B \subset A_1 \times A_2 \subset A^{n_1+n_2}$. Hence,
$A^{n_1+n_2}/B$ can be written in the form $V_1^{s_1} \times \cdots \times V_l^{s_l}$,
where for each $1 \le i \le l$, $s_i \le r_i$. Now $B$ is the kernel
of the natural map from $A^{n_1+n_2} \to A^{n_1+n_2}/B$.
So for this we find the corresponding map
$V_1^{r_1} \times \cdots \times V_l^{r_l} \to V_1^{s_1} \times \cdots \times V_l^{s_l}$,
for which $B$ is isogenous to the kernel.

Such a map is expressed as a block diagonal matrix of order
$(s_1+\cdots+s_l,r_1+\cdots+r_l)$. This matrix has diagonal blocks of order
$(s_i,r_i)$ with entries from $End(V_i)$, for each $1 \le i \le l$.
Now such a matrix acts on an element $(x_{11}, \ldots, x_{1r_1}, \ldots , x_{l1}, \ldots, x_{l r_l})$
of $TA^{n_1+n_2}$, written as a column. Under this action an element of $TB$ is mapped
to the zero vector i.e.
\begin{equation}\label{def-TB-2}
\left( \begin{array}{c c c c}
(\delta_{1jk})_{s_1 \times r_1} & (0)_{s_1 \times r_2} & \cdots & (0)_{s_1 \times r_l}\\
(0)_{s_2 \times r_1} & (\delta_{2jk})_{s_2 \times r_2} & \cdots & (0)_{s_2 \times r_l}\\
\vdots & \vdots & \ddots & \vdots\\
(0)_{s_l \times r_1} & (0)_{s_l \times r_2} & \cdots & (\delta_{ljk})_{s_l \times r_l}
\end{array} \right)
\left( \begin{array}{c} {{\bf x}_{1}} \\ {{\bf x}_{2}}\\ \vdots \\ {{\bf x}_{l}}
\end{array} \right)
 = {\bf 0},
\end{equation}
where for $1 \le i \le l$, ${{\bf x}_{i}}=\left( \begin{array}{c} {x_{i1}} \\ \vdots \\ x_{i r_i}
\end{array} \right)$.
Since the matrix is block diagonal, we get that
a defining relation for $TB$ is given as the relations of the form
$$
\sum_{k=1}^{r_i} \delta_{ijk} x_{ik} = 0
$$
for some $\delta_{ijk} \in End(V_i)$ with $1 \le i \le l, 1\le j \le s_i$ and $1 \le k \le r_i$.
This completes the proof.

\subsection{A special case}

Let us consider the following two sub-fields of $\cE$ and $\cL$ :
$$
\cE' =\bigcup_{n \ge 0} \cE_n' \ \ \text{and} \ \ \cL' =\bigcup_{n \ge 0} \cL_n'
$$
where $\cE_0'=\bQ$, $\cL_0' = \bQ(\Lambda_A)$ and for $n \ge 1$,
$$
\cE_n' = \overline{\cE_{n-1}'(\{\text{components of }\exp_A(u) : u \in \cE_{n-1}^{'g}\})}
$$
and
$$
\cL_n' = \overline{\cL_{n-1}'(\{\text{components of }u : \exp_A(u) \in A(\cL_{n-1}')\})}.
$$
In fact, by induction, one can see that for $n \ge 1$,
$$
\cE_n' = \overline{\Q(\text{components of } \exp_A(\cE_{n-1}^{'g}))}
\ \ \text{and} \ \
\cL_n' = \overline{\Q(\Lambda_A,\text{components of } \exp_A^{-1}(A(\cL_{n-1}')))}.
$$
Arguments as in our proof  of \thmref{lin-dis} immediately yields the following : 

\begin{thm}\label{lin-dis-1}
If \eqref{qp-0} is true, then $\cE'$ and $\cL'$ are linearly disjoint over $\bQ$.
\end{thm}

\section{An application}
%We first show that in the special case of an elliptic curve $E$, defined over $\bQ$,
%\conjref{eS} and \conjref{aS} are equivalent. Taking $A=E^n$, we recover \conjref{eS}
%from \conjref{aS}.
%
%For the other implication, we first consider the CM case. Let $\omega_1$ be a
%fundamental period. Suppose, $u=(u_1,\ldots,u_n) \in \C^n$. Without loss of
%generality let $\{u_1,\ldots, u_m\}$ denote the maximal $K$-linearly independent
%subset of $\{u_1,\ldots,u_n\}$, where $K$ denotes the field of endomorphisms.
%
%Suppose, $\omega_1,u_1,\ldots, u_m$ be $K$-linearly independent. Then
%by \conjref{eS}, we have
%$$
%\trdeg_{\bQ}\bQ(\omega_1,u_1,\ldots, u_m, \wp(\omega_1/2),\wp(u_1),\ldots, \wp(u_m)) \ge m+1.
%$$
%Hence,
%$$
%\trdeg_{\bQ(\Lambda)}\bQ(\Lambda,u_1,\ldots, u_m,\wp(u_1),\ldots, \wp(u_m)) \ge m.
%$$
%If $H$ denotes the smallest algebraic subgroup of $E^n$
%containing the point $(\wp(u_1),\ldots, \wp(u_n))$, then we have
%$\dim(H) \le m$, by using {\it the addition theorem} (see \cite[p. 34]{KC}),
%as $\{u_1,\ldots, u_m\}$ is the maximal $K$-linearly
%independent subset of $\{u_1,\ldots,u_n\}$.
%
%Now if $\omega_1,u_1,\ldots, u_m$ are not $K$-linearly independent.
%Then we can express $\omega_1/2$ as linear combination of $u_1,\ldots, u_m$.
%This gives another defining relation for the subgroup $H$. Therefore,
%$\dim(H) \le m-1$. From \conjref{eS}, we have
%$$
%\trdeg_{\bQ}\bQ(\omega_1,u_1,\ldots, u_m, \wp(\omega_1/2),\wp(u_1),\ldots, \wp(u_m)) \ge m.
%$$
%i.e.
%$$
%\trdeg_{\bQ(\Lambda)}\bQ(\Lambda,u_1,\ldots, u_m,\wp(u_1),\ldots, \wp(u_m)) \ge m-1 \ge \dim(H).
%$$
%The non-CM case is treated similarly.

Now we consider the two recursively defined sets $\cE,\cL$
related to the Weierstrass $\wp$-function associated to a lattice
$\Lambda$ with algebraic invariants $g_2,g_3$, defined as follows
$$
\cE =\bigcup_{n \ge 0} \cE_n \ \ \text{and} \ \ \cL =\bigcup_{n \ge 0} \cL_n
$$
where $\cE_0=\bQ, \cL_0 = \overline{\Q(\Lambda)}$  and for $n \ge 1$,
$$
\cE_n = \overline{\cE_{n-1}(\{\wp(x) : x \in \cP_{n-1} \setminus \Lambda\})}
\ \ \text{and} \ \
\cL_n = \overline{\cL_{n-1}(\{x : \wp(x) \in \cL_{n-1} \cup \{\infty\}\})}.
$$
In fact, by induction, one can see that for $n \ge 1$,
$$
\cE_n = \overline{\Q(\wp(\cE_{n-1} \setminus \Lambda))}
\ \ \text{and} \ \
\cL_n = \overline{\Q(\wp^{-1}(\cL_{n-1} \cup \{\infty\}))}.
$$
As a corollary to \thmref{lin-dis}, we obtain the following result.

\begin{cor}\label{lin-dis-2}
If \conjref{eS} is true, then $\cE$ and $\cL$
are linearly disjoint over $\bQ$.
\end{cor}

In fact, in view of \thmref{lin-dis-1}, the conclusion of the above corollary also holds
if \eqref{weak-elliptic} is true. We know that when $g_2,g_3$ are algebraic,
$\omega \in \Lambda\setminus\{0\}$ is transcendental. Hence as a corollary 
to \corref{lin-dis-2} we obtain that $\omega \notin \cE$.
In particular, $\omega$ can not be a value of the $\wp$ function iterated
at an algebraic point.

\subsection{An example}

We now exhibit an example in the elliptic case which shows that
if the curve is not defined over $\bQ$, then the conclusion of our
\thmref{lin-dis} (or, \corref{lin-dis-2}) does not follow from the same
hypothesis.

For a fixed real algebraic irrational number $\alpha$, our aim is to find out a
lattice $\Z + \Z \tau$ such that $\wp(\tau;\alpha)=\wp(\Z+\Z\tau;\alpha)$
can be written as a polynomial in $\tau$ with algebraic coefficients.

We first claim that we can find $\tau_0$ such that
$\frac{d}{d \tau} \wp(\tau;\alpha) |_{\tau=\tau_0} \neq \frac{\wp(\tau_0;\alpha)}{\tau_0}$.
If not, then we have
$\frac{d \wp(\tau;\alpha)}{\wp(\tau;\alpha)}|_{\tau=\tau_0} = \frac{d \tau}{\tau}|_{\tau=\tau_0}$
for all $\tau_0$ in the complex upper half plane. Thus $d (\log(\wp(\tau;\alpha)))|_{\tau=\tau_0}
=d (\log \tau)|_{\tau=\tau_0}$ i.e. $\wp(\tau_0;\alpha)=c_\alpha \tau_0$ for all $\tau_0$,
where $c_\alpha$ is a constant depending on $\alpha$. But this is not possible, as can be
checked from the $q$-expansion of $\wp$.

So we choose $\tau_0$ such that
$\frac{d}{d \tau} \wp(\tau;\alpha) |_{\tau=\tau_0} \neq \frac{\wp(\tau_0;\alpha)}{\tau_0}$
and denote the ratio $\frac{\wp(\tau_0;\alpha)}{\tau_0}$ by $\lambda_0$. We now choose
$\lambda$ close to $\lambda_0$ such that $\lambda \in \bQ$ and
$\frac{d}{d \tau} \wp(\tau;\alpha) |_{\tau=\tau_0} \neq \lambda$.
Consider the function $f(\tau)=\wp(\tau;\alpha)- \lambda \tau$. Then we get that
$f'(\tau)|_{\tau=\tau_0} \neq 0$. Hence $f$ has a local inverse at $\tau_0$,
say $g$, which is defined in a neighbourhood of $f(\tau_0)$.
Choose $\beta \in \bQ$ sufficiently close to $f(\tau_0)$ and set $\tau_1=g(\beta)$.
Then $\beta=f(\tau_1)=\wp(\tau_1;\alpha)- \lambda \tau_1$. Thus $\wp(\tau_1;\alpha)$
can be written as a polynomial in $\tau_1$ with algebraic coefficients.

The elliptic Schanuel conjecture implies that $\tau_1$ is transcendental.
Indeed, if $\tau_1$ is a quadratic irrational, then the associated $j$ invariant 
is algebraic. Hence $g_2,g_3$ are algebraically related and 
$\eta_1,\eta_2$, satisfying Masser's relation, are algebraically dependent
over $\overline{\Q(g_2,g_3)}$. Now from \conjref{eS} we get a contradiction by
taking $n=1$ and $x_1=\alpha$. If $\tau_1$ is algebraic of degree
larger than 2, then \conjref{eS} gives a contradiction again for $n=1$ and
$x_1=\alpha$.

Now for this choice of $\tau_1$, we see that $\tau_1$ belongs to both $\cE_1$ and $\cL_1$,
where the tower of fields $\cE_i$'s and $\cL_i$'s are constructed as in the beginning of
this section, but with $\bQ$ replaced by the corresponding field of definition $\overline{\Q(g_2,g_3)}$.
However, we show below that $\tau_1$ is transcendental over $\Q(g_2,g_3)$.
This gives $\trdeg_{\cE_0} \cE_1 \cap \cL_1 \ge 1$, which implies
$\cE_1 \cap \cL_1 \neq \cE_0$ and therefore $\cE_1$ and $\cL_1$
are not linearly disjoint over $\cE_0$.

To prove that $\tau_1$ is transcendental over $\Q(g_2,g_3)$,
note that $1,\tau_1$ and $\alpha$ are $\Q$ linearly independent. Then the
elliptic Schanuel conjecture yields
$\trdeg_\Q \Q(g_2,g_3,1,\tau_1,\alpha,\wp(\tau_1;\alpha)) \ge 3$.
Now from our construction we see that 
$\trdeg_\Q \Q(g_2,g_3,1,\tau_1,\alpha,\wp(\tau_1;\alpha))
=\trdeg_\Q \Q(g_2,g_3,\tau_1)$. Thus $\tau_1$ is transcendental
over $\Q(g_2,g_3)$.

\medskip

\noindent {\bf Acknowledgement:}
The authors would like to thank D. Bertrand for helpful
discussions and useful comments on an earlier version of this article.
The second and the third author would like to thank the Institut de
Math\'ematiques de Jussieu for hospitality during academic visits in the
frame of the IRSES Moduli and LIA. Research of the second author is also
supported by SERB-DST-NPDF grant vide PDF/2016/002938. The authors
would also like to thank the referee for helpful suggestions and Luca Ghidelli
for pointing out an inaccurate assumption in the published version.


\begin{thebibliography}{100}

\bibitem{YA1}
Y. Andr\'e, {\it Mumford-Tate groups of mixed Hodge structures and
the theorem of the fixed part}, Compositio Math. {\bf 82} (1992), no. 1, 1--24. 


\bibitem{YA}
Y. Andr\'e, Une introduction aux motifs, Panoramas et Synth\`eses {\bf 17},
{\it Soci\'et\'e Math. France}, Paris, (2004).

\bibitem{CB}
C. Bertolin, {\it P\'eriodes de 1-motifs et transcendance},
J. Number Theory {\bf 97} (2002), no. 2, 204--221. 

\bibitem{KC}
K. Chandrasekharan, Elliptic functions,
Grundlehren der Mathematischen Wissenschaften 281,
{\it Springer-Verlag}, Berlin, 1985.

\bibitem{CDHHKMMMW}
C. Cheng, B. Dietel, M. Herblot, J. Huang, H. Krieger, D. Marques,
J. Mason, M. Mereb and  S.R. Wilson,
{\it Some consequences of Schanuel's conjecture},
{J. Number Theory} {\bf 129} (2009), no. 6, 1464--1467. 

\bibitem{PD}
P. Deligne, {\it Hodge cycles on abelian varieties},
{Hodge cycles, motives and Shimura varieties},
Lecture Notes in Math. {\bf 900}, {\it Spinger-Verlag}, New York, (1975), 6--77.

\bibitem{SL}
S. Lang, Algebra, Graduate Texts in Math. 211 (Rev. 3rd Ed.),
{\it Springer-Verlag}, New York, (2002).

\bibitem{PP}
P. Philippon, {\it Vari\'et\'es ab\'eliennes et ind\'ependance alg\'ebrique II:
Un analogue ab\'elien du th\'eor\`eme de Lindemann-Weierstra\ss},
Invent. Math. {\bf 72} (1983), no. 3, 389--405. 

\bibitem{GV}
G. Vall\'ee, {\it Sur la conjecture de Schanuel ab\'elienne}, preprint.


\bibitem{GW}
G. W\"ustholz, {\it \"Uber das Abelsche Analogon des Lindemannschen Satzes I},
Invent. Math. {\bf 72} (1983), no. 3, 363--388.

\end{thebibliography}
\end{document}